\documentclass[12pt]{amsart}

\usepackage{amsmath, amscd, graphicx, latexsym, hyperref, rlepsf, times, }

\oddsidemargin .25in \theoremstyle{plain}

\newtheorem{Thm}{Theorem}

\newtheorem*{Thm*}{Theorem}

\newtheorem{Lem}[Thm]{Lemma}

\newtheorem{Cor}[Thm]{Corollary}

\newtheorem{Prop}[Thm]{Proposition}

\newtheorem{Def}[Thm]{Definition}

\newtheorem{Rem}[Thm]{Remark}

\def\CP{\hbox{${\Bbb C} P^2$}}

\def\CPb{\hbox{$\overline{{\Bbb C}P^2}$}}

\def\v{\vskip.12in}

\begin{document}

\title[]{Exotic Stein fillings with arbitrary fundamental group}

\author[Akhmedov and Ozbagci]{Anar Akhmedov and Burak Ozbagci}

\address{School of Mathematics, University of Minnesota, Minneapolis, MN 55455, USA, \newline akhmedov@math.umn.edu}

\address{Department of Mathematics, Ko\c{c} University, Istanbul,
Turkey, \newline bozbagci@ku.edu.tr}

%\subjclass[2000]{57R17}

%\date{\today}

\begin{abstract}

For any finitely presentable group $G$, we show the existence of an
isolated complex surface singularity link which admits infinitely
many exotic Stein fillings such that the fundamental group of each
filling is isomorphic to $G$. We also provide an infinite family of
closed exotic smooth four-manifolds with the fundamental group $G$
such that each member of the family admits a non-holomorphic
Lefschetz fibration over the two-sphere.

\end{abstract}

\maketitle

\section{Introduction}

Although many examples of isolated complex surface singularity links
which admit only finitely many Stein (or symplectic) fillings have
appeared in the literature,  very few examples of singularity links
with infinitely many \emph{exotic} (i.e., homeomorphic but pairwise
non-diffeomorphic) Stein fillings are known. Until 2007 no exotic
Stein fillings were known and the first such examples were
constructed by Akhmedov et al. in \cite{aems}. This paper is a
sequel to \cite{ako} in which the authors have shown that, for each
$m \geq 1$, there exists a (Seifert fibered) singularity link which
admits infinitely many exotic Stein fillings of its canonical
contact structure such that the fundamental group of each filling is
$\mathbb{Z} \oplus \mathbb{Z}_m$, extending the result in
\cite{aems}. The main goal of this paper is to improve this result
by replacing the group $\mathbb{Z} \oplus \mathbb{Z}_m$ with an
arbitrary finitely presentable  group $G$. More precisely, we prove that \\

\begin{Thm} \label{main1} For any finitely presentable
group $G$, there exists an isolated complex surface singularity link
which admits infinitely many exotic Stein fillings of its canonical
contact structure such that the fundamental group of each filling is
isomorphic to $G$.
\end{Thm}

%It is interesting to compare our result above with the one obtained
%recently by J. Kollar and M. Kapovich in \cite{kk} (cf. Theorem 1, page 1)
%where complex $3$-dimensional case have been studied. For every finitely
%presented group $G$ they prove there is an isolated, 3-dimensional,
%complex singularity link with the fundamental group $G$. Their proof uses
%techniques of algebraic geometry, and is quite different from the approach
%taken in this paper.
We also prove that

\begin{Thm} \label{main2} For any finitely presentable
group $G$, there exists an infinite family of closed exotic
symplectic $4$-manifolds with the fundamental group $G$ such that
each member of this family admits a non-holomorphic Lefschetz
fibration over $S^2$.
\end{Thm}

%In order to prove Theorem~\ref{main1} we use Lefschetz fibrations, a
%fundamental instrument in four dimensional symplectic topology
%albeit it never appears in the statement.

In his ground-breaking work, Donaldson \cite{don} proved that every
closed symplectic 4-manifold admits a Lefschetz pencil over the
$2$-sphere and Gompf \cite{Gom} showed that every finitely
presentable group $G$ can be realized as the fundamental group of
some closed symplectic $4$-manifold. Since a Lefschetz pencil can be
turned into a Lefschetz fibration by blowing up its base
locus---that has no effect on the fundamental group of the
underlying $4$-manifold---one immediately obtains the existence of a
closed symplectic $4$-manifold with fundamental group $G$, which
admits a Lefschetz fibration over $S^2$.

Since the total space of a Lefschetz fibration with fiber genus
greater than one admits a symplectic structure \cite{gs}, the result
above can also be proven by constructing an explicit Lefschetz
fibration over $S^2$ whose total space has fundamental group $G$
(cf. \cite{ABKP} and \cite{kor}).

In this paper, using \emph{Luttinger surgery}, we describe a new
construction of a closed symplectic $4$-manifold with $\pi_1=G$ and
$b_2^+ \geq 2$ which admits a Lefschetz fibration over $S^2$ that
has two additional features: (i) the symplectic $4$-manifold
contains a homologically essential embedded torus of square zero
intersecting each fiber of the Lefschetz fibration twice and (ii)
the Lefschetz fibration admits many $(-2)$-sphere sections disjoint
from this embedded torus. Moreover, by performing the knot surgery
operation along on such a homologically essential torus, we obtain
an infinite family of non-holomorphic \emph {exotic Lefschetz
fibrations} over $S^2$ with $\pi_1=G$---which is the content of
Theorem~\ref{main2}. Once we have these results at hand, we follow a
strategy similar to the one employed in \cite{ako} to prove
Theorem~\ref{main1}.

We would like to emphasize that none of the previous constructions
of Lefschetz fibrations in the literature could be effectively
utilized to prove the main results of this article.  For example,
the Lefschetz fibrations over $S^2$ described in \cite{kor} do not
carry the homologically essential tori we need for producing  exotic
copies of the Stein fillings using the knot surgery operation. This
is due to the fact that the examples in \cite{kor} obtained by
performing many symplectic sums along higher genus surfaces, in
contrast to the examples presented in this article, where we perform
\emph{only two} symplectic sums. Moreover, the \emph{separating
vanishing cycles} appearing in those Lefschetz fibrations do not
allow to prove the Steinness of the remaining piece after the
removal of some sections and a fiber.

The direct approach using Donaldson's Lefschetz pencils would not
work for us either, since any of the sections of a Lefschetz
fibration obtained by blowing up the base locus of a Lefschetz
pencil has self-intersection $-1$ and therefore is not suitable for
our construction of exotic Stein fillings of an isolated complex
surface singularity. Furthermore, the additional feature (i) listed above
are absolutely crucial for the purposes of this paper. The existence of such
tori serve to preserve a certain horizontal Lefschetz fibration structure after
the knot surgery operation.

%the base locus of a Lefschetz pencil on a closed symplectic $4$-manifold induces a section
%of the Lefschetz fibration after blowing up, but a priori there is no control
%on the number of points in the base locus once the symplectic $4$-manifold is given.

We would like to point out that the total space of any of the Lefschetz fibrations that we construct in this paper is
symplectically minimal, which follows from Usher's theorem in \cite{Us}, and has $b_2^+ \geq 2$. The case $b_2^+ = 1$ has been studied separately in \cite{AZ}.
These later examples, however, do not necessarily yield Stein fillings.

\section{Luttinger surgery and symplectic sum}

Luttinger surgery (cf. \cite{Lu}, \cite{ADK}) has been a very
effective tool recently for constructing exotic smooth structures on
$4$-manifolds. In this section, we briefly recall Luttinger surgery
and symplectic sum operations.

\subsection{Luttinger surgery} \label{subsec:Luttinger}
Let $L$ be a Lagrangian torus embedded in a closed symplectic
$4$-manifold $(X, \omega)$. It follows that $L$ has a trivial normal
bundle. In addition, by the Lagrangian neighborhood theorem of
Weinstein, a neighborhood $\nu L$ of $L$ in $X$ can be identified
\emph{symplectically} with a neighborhood of the zero-section in the
cotangent bundle $T^*L \simeq T \times \mathbb{R}^2$ with its
standard symplectic structure. Let $\gamma$ be any simple closed
curve on $L$. The Lagrangian framing described above determines, up
to homotopy, a push-off of $\gamma$ in $\partial(\nu L)$, which we
denote by $\gamma'$.

\begin{Def}  For any integer $m$, the $(L,\gamma,m)$ \emph{Luttinger
surgery}\/ on $X$\/ is defined as
$$X_{L,\gamma,m} = ( X - \nu L ) \cup_{\phi} (S^1 \times S^1
\times D^2),$$ where, for  a meridian  $\mu_{L}$ of $L$,  the gluing map $\phi : S^1 \times S^1 \times \partial D^2 \to
\partial(X - \nu L )$ satisfies
$\phi([\partial D^2]) =[\mu_{L}] + m[\gamma']$ in $H_{1}(\partial(X
- \nu L)$.
\end{Def}

\begin{Rem}
A salient feature of Luttinger surgery is that it can be done
symplectically, i.e., the symplectic form $\omega$ on $X - \nu L$
can be extended to a symplectic form on $X_{L,\gamma,m}$ as shown in
\cite{ADK}.
\end{Rem}

\begin{Lem}\label{LSL} We have $\pi_1(X_{L,\gamma,m}) = \pi_1(X- \nu L)/N(\mu_{L}
\gamma'^m)$, where $N(\mu_{L} \gamma'^m)$ denotes the normal
subgroup of $\pi_1(X- \nu L)$ generated by the product $\mu_{L}
\gamma'^m$. Moreover, we have $\sigma(X)=\sigma(X_{L,\gamma,m})$,
and $\chi(X)=\chi(X_{L,\gamma,m})$, where $\sigma$ and $\chi$ denote
the signature and the Euler characteristic, respectively.
\end{Lem}

\begin{proof}
The result about the fundamental group follows from the Seifert-van
Kampen's theorem, while the signature formula is just Novikov
additivity. The result about the Euler characteristics is evident.
\end{proof}

\subsection{Symplectic sum}

Let $(X_1, \ \omega_1)$ and $(X_2, \ \omega_2)$ be closed symplectic
$4$-manifolds containing closed embedded surfaces $F_1$ and $F_2$ of
genus $g$, with normal bundles $\nu_1$ and $\nu_2$, respectively.
Assume that $e(\nu_1) + e(\nu_2) = 0$, where $e(\nu_i)$ denotes the
Euler class of the bundle $\nu_i$.

\begin{Def} For any choice of a fiber-orientation reversing bundle isomorphism
$\psi: \nu_1 \to \nu_2$, the \emph{symplectic sum} of $X_1$ and
$X_2$ is defined as the closed $4$-manifold $$X_1 \#_{\psi} X_2 =
(X_1 - \nu_1)\cup_{\psi} (X_2 -\nu_2).$$

\end{Def}

This gluing is called a symplectic sum since there is a natural
isotopy class of symplectic structures on $X_1 \#_{\psi} X_2$
extending the symplectic structures on $X_1 -
\nu_1$ and $X_2 -\nu_2$ as shown in \cite{Gom}.

%In addition, it is easy to see that\\

%(i) $\sigma(X_1 \#_{\psi} X_2)=\sigma(X_1)+ \sigma (X_2)$, and

%(ii) $\chi(X_1 \#_{\psi} X_2)=\chi(X_1)+\chi(X_2)+4g-4.$

\section{Lefschetz fibrations with arbitrary fundamental group}

In this section, for each finitely presentable group G,  we
construct a closed symplectic $4$-manifold with $\pi_1=G$ and $b_2^+
\geq 2$ which admits a Lefschetz fibration over $S^2$ having the
additional properties (i) and (ii) listed in the introduction. Parts
of this section overlaps with certain parts of \cite{AZ}, where the
case $b_2^+ = 1$ has been studied. In the following, we first
explain our construction for the case of a finitely generated free
group, before we deal with the general case.

\subsection{Construction for a finitely generated free group} \label{fund}

The product $\Sigma_{g}\times \mathbb{T}^2$ admits a symplectic
structure, where $\Sigma_{g}$ and  $\mathbb{T}^2$ denote a closed
symplectic genus $g$ surface and a symplectic torus, respectively.
Let $\{p_i, q_i \geq 0  : 1 \leq i \leq g \}$ be a set of
nonnegative integers and let $\overline{p}=(p_1, \ldots, p_g)$ and
$\overline{q}= (q_1, \ldots, q_g)$.  We denote by
$M_{g}(\overline{p}, \overline{q})$ the symplectic $4$-manifold
obtained by performing a Luttinger surgery on the symplectic
$4$-manifold $\Sigma_{g}\times \mathbb{T}^2$ along each of the $2g$
Lagrangian tori with the associated framings belonging to the set

$$\mathcal{L}=\{(a_i' \times c', a_i', -p_{i}), (b_i' \times c'', b_i',
-q_{i}) \; | \; 1 \leq i \leq g\},$$

\noindent where $\{a_i,b_i : 1 \leq i \leq g\}$ is the set of
standard generators of $\pi_1(\Sigma_{g})$ and $c,d$\/ are the
standard generators of $\pi_1(\mathbb{T}^2)$. This family of
symplectic $4$-manifolds $M_{g}(\overline{p}, \overline{q})$ have
been studied %by the first author and B. Doug Park
in \cite{AP} (see discussion on pages 2--3, 13--14). For further details, we refer the
reader to \cite{AP} and references therein. The proof of the
following result essentially follows from the Example on page 189 in
\cite{ADK}.

\begin{Lem}\label{fiber} The $4$-manifold $M_{g}(\overline{p}, \overline{q})$
admits a locally trivial genus $g$ fibration over $\mathbb{T}^2$.

\end{Lem}

\begin{proof}  The $(a_i' \times c', a_i', -p_{i})$ or $(b_i' \times c'', b_i', -q_{i})$ Luttinger
surgery in the trivial bundle $\Sigma_{g}\times \mathbb{T}^2$
preserves the fibration structure over $\mathbb{T}^2$ introducing a
monodromy of the fiber $\Sigma_{g}$ along the curve $c$ in the base.
Depending on the type of the surgery the monodromy is either
$(t_{a_i})^{p_i}$ or $(t_{b_i})^{q_i}$, where $t$ denotes a Dehn
twist.

\end{proof}

An alternative proof of the above lemma can be obtained by
identifying the symplectic $4$-manifold $M_{g}(\overline{p},
\overline{q})$ with $M_{K} \times S^{1}$, where $M_{K}$ is the
$3$-manifold obtained by $0$-surgery along a suitably chosen fibered
genus $g$ knot $K$ in $S^{3}$ (see for example the symplectic
building blocks in \cite{Ak} via knot surgery and their construction
using Luttinger surgery in appendix of \cite{ABP}).

%Notice that the order of the Luttinger surgeries determines the
%monodromy of the fibration in Lemma~\ref{fiber}.

The proof of the next result---which is essentially a consequence of
Lemma~\ref{LSL}---can be found in \cite{AP}.

\begin{Cor}\label{funda}

The fundamental group of $M_{g}(\overline{p}, \overline{q})$ is
generated by $a_i,b_i$ $(i=1,\ldots, g)$ and $c,d$,
with the following relations:\\

$(1)$ $[b_i^{-1},d^{-1}]=a_i^{p_{i}},\ \ [a_i^{-1},d]=b_i^{q_{i}}$,
for all $1 \leq i \leq g$,

$(2)$ $[a_i,c]=1,\ \  [b_i,c]=1$, for all $1 \leq i \leq g$,

$(3)$ $[a_1,b_1][a_2,b_2]\cdots[a_g,b_g]=1$, and

$(4)$ $[c,d]=1$. \\

\end{Cor}

The torus $\{{\rm pt}\} \times \mathbb{T}^2 \subset \Sigma_{g}\times
\mathbb{T}^2$ induces a torus $T$ with trivial normal bundle in
$M_{g}(\overline{p}, \overline{q})$. On the other hand,  a regular
fiber of the elliptic fibration on the complex surface $E(n)$ is
also a torus of square zero.

\begin{Def} Let $X_{g,n}(\overline{p}, \overline{q})$ denote the symplectic sum of $M_{g}(\overline{p}, \overline{q})$
along the torus $T$ with the elliptic surface $E(n)$ along a regular
elliptic  fiber.
\end{Def}

\begin{Lem} \label{lef} The symplectic $4$-manifold $X_{g,n}(\overline{p},
\overline{q})$ admits a genus $2g+n-1$ Lefschetz fibration over
$S^2$ with at least $4n + 4$ pairwise disjoint sphere sections of
self intersection $-2$. Moreover, $X_{g,n}(\overline{p},
\overline{q})$ contains a homologically essential embedded torus of
square zero disjoint from these sections which intersects each fiber
of the Lefschetz fibration twice.

%Moreover, $\chi(X_{g,n}(\overline{p}, \overline{q}))=12n$, and
%$\sigma(X_{g,n}(\overline{p}, \overline{q}))=-8n$.

\end{Lem}

\begin{proof}
By definition,  $X_{g,n}(\overline{p}, \overline{q})$ is obtained as
the symplectic sum of the complex surface $E(n)$ along a regular
elliptic fiber with the symplectic $4$-manifold $M_{g}(\overline{p},
\overline{q})$ along a torus $T$ which is induced from $\{ pt \}
\times \mathbb{T}^2$ in $\Sigma_g \times \mathbb{T}^2$. Notice that
by Lemma~\ref{fiber}, $M_{g}(\overline{p}, \overline{q})$ is a
locally trivial genus $g$ bundle over $\mathbb{T}^2$ where $T$ is a
section.

On the other hand, since the complex surface $E(n)$ can be obtained
as a desingularization of the branched double cover of $S^2 \times
S^2$ with the branching set being $4$ copies of $\{pt\} \times S^2$
and $2n$ copies of $S^2 \times \{pt\}$, it admits a genus $n-1$
fibration over $S^2$ as well as an elliptic fibration over $S^2$,
both of which are obtained by the projection of $S^2 \times S^2$
onto one of the $S^2$ factors. In fact, both fibrations can be
realized as Lefschetz fibrations and a regular fiber of the elliptic
fibration on $E(n)$ intersects every genus $n-1$ fiber of the other
Lefschetz  fibration twice.

Consequently, when performing the symplectic sum of $E(n)$ along a
regular elliptic fiber with the surface bundle $M_{g}(\overline{p},
\overline{q})$ along a torus section $T$,  the fibration structures
in both pieces can be glued together to yield a genus $2g + n -1$
Lefschetz fibration on $X_{g,n}(\overline{p}, \overline{q})$ over
$S^2$.

In addition, we observe that a sphere section of the genus $n-1$
Lefschetz fibration $E(n) \to S^2$ induce a section of the genus $2g
+ n - 1$ Lefschetz fibration $X_{g,n}(\overline{p}, \overline{q})
\to S^2$. Since $E(n)$ can be realized as a fiber sum of two copies
of $\CP \#(4n+5)\CPb$ along a genus $n-1$ surface, and since there is
a Lefschetz fibration $\CP \#(4n+5)\CPb \to S^2$ with at least $4n +
4$ pairwise disjoint sphere sections of self intersection $-1$ (cf.
\cite{t}), we conclude that the genus $n-1$ Lefschetz fibration $
E(n) \to S^2$ has at least $4n + 4$ pairwise disjoint sphere
sections of self intersection $-2$.

Moreover, the homologically essential embedded torus $T$ of square
zero in $X_{g,n}(\overline{p}, \overline{q})$ which is disjoint from
these sections intersects each fiber of the Lefschetz fibration
twice.

%Moreover, since $\chi(\Sigma_{g}\times
%\mathbb{T}^2)=\sigma(\Sigma_{g}\times \mathbb{T}^2)=0$ and Luttinger
%surgery preserves these characteristic numbers by Lemma~\ref{LSL},
%we conclude that $\chi(M_g (\overline{p}, \overline{q}))=0$ and
%$\sigma(M_g (\overline{p}, \overline{q}))=0$. Hence we get
%$\chi(X_{g,n}(\overline{p}, \overline{q}))=12n$, and
%$\sigma(X_{g,n}(\overline{p}, \overline{q}))=-8n$, since these
%characteristic numbers are additive under symplectic sum along a
%\emph{torus}.

\end{proof}

\begin{Lem}\label{pres} The fundamental group of the symplectic $4$-manifold $X_{g,n}(\overline{p}, \overline{q})$ is
generated by the set $\{a_i,b_i : 1 \leq i \leq g\}$ subject to the
relations: \\

$(1)$ $a_i^{p_i} = 1,b_i^{q_i} = 1$, for all $1 \leq i \leq g$, and

$(2)$ $\Pi_{j=1}^{g} [a_j, b_j] = 1.$

\end{Lem}

\begin{proof}

Choose a base point $x$ on $\partial(\nu T)$ such that
$\pi_1(M_{g}(\overline{p}, \overline{q})\setminus\nu T,x)$ is
normally generated by $a_i,b_i$ ($i=1, \cdots, g$) and $c,d$. Notice that
the symplectic torus $\{ pt \}\times \mathbb{T}^2$ is disjoint from the
neighborhoods of $2g$ Lagrangian tori in $\mathcal{L}$ above.
Consequently, all but relation (3) in Corollary~\ref{funda} holds in
$\pi_1(M_{g}(\overline{p}, \overline{q})\setminus\nu T)$. The
product $[a_1,b_1][a_2,b_2]\cdots[a_g,b_g]$ is no longer trivial,
and it represents a meridian of $T$ in $\pi_1(M_{g}(\overline{p},
\overline{q}) \setminus\nu T)$. Since $\pi_1(E(n) \setminus (\nu(T))
= 1$, after the fiber sum we have $c = d = 1$ in the fundamental
group of $X_{g,n}(\overline{p}, \overline{q})$. Hence we obtain the
desired presentation for $\pi_1 (X_{g,n}(\overline{p},
\overline{q}))$.

\end{proof}

\begin{Cor} By setting $p_i =1$ and $q_i = 0$, for all $1 \leq i \leq g$,
we observe that the fundamental group of $X_{g,n}((1,1,
\ldots,1),(0,0,\ldots,0))$ is a free group of rank $g$.
%There exists a closed symplectic $4$-manifold whose fundamental
%group is isomorphic to a given finitely generated free group.
\end{Cor}

%Given any finitely presented group $G$, the arguments above can be
%generalized using some additional Luttinger surgeries to obtain a
%closed symplectic $4$-manifold $X_n(G)$ whose fundamental group is
%isomorphic to $G$---which gives a new proof of Gompf's theorem \cite{Gom}.

\subsection{Construction for an arbitrary finitely presentable group}

Let $G$ be a finitely presentable group with a given presentation
$\langle x_{1}, \dots, x_{k} \ | \ r_{1}, \dots, r_{s} \rangle.$ The
syllable length $l(w)$ of a word $w \in G$ in the letters $x_{1},
\dots, x_{k}$ is defined as

$$ l(w)= \{ \min m \;|\; w = x_{i_1}^{n_1} x_{i_2}^{n_2} \cdots
x_{i_m}^{n_m}, 1 \leq i_{j} \leq k, n_{j} \in \mathbb{Z} \}.$$

Let $\{a_j,b_j : 1 \leq j \leq k\}$ denote the set of standard
generators of $\pi_1(F)$, where $F$ is a closed genus $k$ surface.
Since $r_i$ is a word in the generators $x_1, \ldots, x_k$, there is
a smooth immersed oriented circle $\gamma_i$ on $F$ representing the
corresponding word in $\pi_1(F)$, obtained by replacing each $x_j$
with $b_j$. We may choose the loop $\gamma_i$ (up to homotopy) such
that at each self-intersection point, only two segments of
$\gamma_i$ intersect transversely.  In order to carry out some
Luttinger surgeries we have in mind, we first need to resolve the
self-intersection points of $\gamma_{i}$ by a trick that was
initially introduced in \cite{ABKP}, and refined further in
\cite{kor}. We will use the later version below.

For each self-intersection point of the immersed curve $\gamma_i$
where two segments intersect locally, we glue a $1$-handle to $F$
and modify $\gamma_i$ so that one of the intersecting segments goes
through the handle while the other remains under it. The modified
curve on the new surface will be denoted by $\gamma_i$ as well.
Notice that the number of handles needed to resolve all the
self-intersections of $\gamma_i$ is $l(r_i) - 1$. Thus, the total
number of handles attached will be $l - s$, where we just set $l =
l(r_1) + \cdots + l(r_s)$. After these handle additions,  the
surface $F$ is changed to a surface $F'$ of genus $k' = k + (l-s)
\geq k$ so that each (modified) $\gamma_i$ is now an embedded curve
on $F'$.

%\begin{figure}[H]
%\centering
%\includegraphics[scale=.43]{bridge3.eps}
%\caption{Handles on $\Sigma_3$ for the relation
%$b_{1}b_{2}b_{3}b_{2}b_{1}$} \label{fig:bridge1}
%\end{figure}

Let $\{a_i,b_i : 1 \leq j \leq k'\}$ represent the set of standard
generators of $\pi_1(F')$, extending the standard generators of
$\pi_1(F)$. We perform Luttinger surgeries on the standard
symplectic $4$-manifold $F' \times \mathbb{T}^2$, along the
following Lagrangian tori
$$\{(a_i' \times c', a_i', -1), \; (b_i' \times c'', b_i', -1), \ \
k+1\le i\le k'\}.$$
%Hence we again choose similar tori with $p_{k+1}=q_{k+1}=\cdots=p_{k'}=q_{k'}=1$.

Moreover, just as in the free group case, we perform Luttinger
surgeries on $F' \times \mathbb{T}^2$  along the $2k$ Lagrangian
tori belonging to the set

$$\{(a_i' \times c', a_i', -1), \; (b_i' \times c'', b_i',
0) \; | \; 1 \leq i \leq k\}.$$

Let $M(G)$ denote the symplectic $4$-manifold obtained by the total
of $2k'$ Luttinger surgeries on $F' \times \mathbb{T}^2$. Notice
that $k$ of these Luttinger surgeries have a surgery coefficient
$0$, thus have no effect on the associated Lagrangian tori. As in
Section~\ref{fund}, we take a symplectic sum of $M(G)$ along the
torus $T$ descending from $pt \times \mathbb{T}^2$ with $E(n)$ along
a regular elliptic fiber (here we assume $n \geq 2$ for reasons
which will be clear in Section~\ref{ES}), and denote the resulting
symplectic $4$-manifold by $Y_n(G)$. Note that $\pi_1(Y_n(G))$ is a
free group of rank $k$.

Since a regular fiber of a genus $n-1$ hyperelliptic Lefschetz
fibration on $E(n)$ intersect a regular fiber of an elliptic
fibration on $E(n)$ at two points, $Y_n(G)$ admits a genus $2k' +
n-1$ Lefschetz fibration over $S^2$ (see the proof of
Lemma~\ref{lef}).

Note that our manifold $Y_n(G)$ can also be constructed as the twisted fiber sum
of two copies of a genus $2k' + n-1$ Lefschetz fibration on
$\Sigma_{k'}\times S^2\,\#4n\CPb$, where the later family of Lefschetz
obtained from the positive Dehn twist expressions of a certain involution
of the genus $2k'+n -1$ surface in the mapping class group (see \cite{Gu} and \cite{Yu}).
This essentially follows from the fact that the symplectic sum of $E(n)$ along a regular elliptic
fiber with $\Sigma_{k'} \times \mathbb{T}^2$ along a natural square
zero torus is diffeomorphic to the untwisted fiber sum of two copies
of the genus $2k' + n-1$ fibration on $\Sigma_{k'}\times
S^2\,\#4n\CPb$, which in turn follows from the branched cover
description of these $4$-manifolds. When performing the Luttinger
surgeries, the gluing diffeomorphism of the genus $2k' + n - 1$
fibration, which is an identity map initially, turns into the
product of a certain Dehn twists. This gluing $\phi$ diffeomorphism
can be described explicitly using the curves along which we perform
our Luttinger surgeries: $\phi = t_{a_{1}} \cdots t_{a_{k}}
t_{a_{k+1}} t_{b_{k+1}} \cdots t_{a_{k'}}t_{b_{k'}}$.

The global monodromy of the genus $2k' + n -1$ Lefschetz fibration
on $Y_{n}(G)$ is given by the following word: $\theta^{2} \phi^{-1}
\theta^2 \phi = 1$, where $\theta^{2}$ is the global monodromy of
Gurtas' fibration \cite{Gu, Yu}, obtained by factorizing an
involution in the mapping class group of a closed surface of genus
$2k'+n-1$ in terms of positive Dehn twists. The aforementioned
involution $\theta$ is obtained by gluing the hyperelliptic
involution of the genus $n-1$ surface with the vertical involution
of the genus $2k'$ surface with two fixed points \cite{Gu}. %We will study the details of this alternative approach in a subsequent paper.

%over $\mathbb{T}^2$ constructed as a result of some Luttinger
%surgeries on a trivial symplectic bundle $\Sigma_g \times
%\mathbb{T}^2$.

%These family of genus $2k' + n - 1$ Lefschetz fibrations on $\Sigma_{k'}\times \mathbb{S}^{2}\,\#4n\CPb$ has been studied in \cite{Yu}
%by factorizing a set of involutions in the mapping class group of a closed surface of genus $2k'+n-1$ in terms of positive Dehn twists.

In addition to applying the above Luttinger surgeries on the
``standard" tori, we apply $s$ more surgeries along tori:

$$\{ (\gamma_i' \times c''', \gamma_i', -1), \; 1 \leq i \leq s \}.$$

\noindent These tori descend from $M(G)$ and survive in $Y_n(G)$
after the fiber sum with $E(n)$. Let $X_n(G)$ denote the symplectic
$4$-manifold obtained after performing these Luttinger surgeries in
$Y_n(G)$.

%Notice these $m$ Luttinger surgeries will not change the Betti numbers by Lemma \ref{3case}.

In the fundamental group of $X_n(G)$ we have the following
relations---which we  explain below---that come from the last set of
Luttinger surgeries, $$[e_{k_1}^{-1},d]=\gamma_1, \cdots,
[e_{k_s}^{-1},d]=\gamma_s $$ where $e_{k_i}\times d$ is a dual torus
of $\gamma_i'\times c'''$. Here each $e_{k_i}$ is a carefully chosen
disjoint vanishing cycle of Gurtas' genus $2k' + n - 1$ fibration in
\cite{Gu} (see also \cite{Yu}), coming from the hyperelliptic part
of the involution $\theta$, and each $\gamma_i'$ is modified so that
it intersects the vanishing cycles $e_{k_i}$ in a single point.
Since after fiber summing with $E(n)$, we have $c=d=1$, it follows
that $\pi_1(X_{n}(G))$ admits a presentation with generators $\{a_i,
b_i : 1 \leq i \leq k' \}$ and relations:
\begin{gather}
a_1 = 1, %\nonumber  \\
\ldots, %\nonumber \\
%\cdots, \nonumber \\
a_{k'}= 1, \nonumber \\
b_{k+1} = 1, %\nonumber  \\
%\cdots, \nonumber \\
\ldots, %\nonumber \\
b_{k'}= 1, \nonumber \\
\gamma_1 = 1, %\nonumber \\
\ldots, %\nonumber \\
%\cdots, \nonumber \\
\gamma_{s} = 1. \nonumber \\
%\Pi_{j=1}^{g} [a_j, b_j] = 1, \nonumber \\
\nonumber
\end{gather}

In other words, $\pi_1(X_{n}(G)) =  \langle b_{1},  \dots, b_{k}  \;
| \; \gamma_{1},  \dots,  \gamma_{s} \rangle$, which is indeed
isomorphic to the given group $G$. The above presentation follows
from the following facts: (i) $c = d = 1$ in $\pi_{1}(X(G))$, (ii)
for each torus $T_{i} = \gamma_i' \times c'''$ there is at least one
vanishing cycle $e_{k_i}$ of the genus $2k' + n - 1$ Lefschetz
fibration on $Y_n(G)$ (see Theorem 2.0.1, page 12, \cite{Gu}) such
that $\gamma_i'$ intersects $e_{k_i}$  precisely at one point, and
$\gamma_i'$ does not intersect with $e_{k_j}$ for any $j \neq i$.

After the $s$ Luttinger surgeries on tori $T_{i} = \gamma_i' \times c'''$,
 we obtain the following set of relations: $$[e_{k_1}^{-1},d]=\gamma_1, \cdots,
[e_{k_s}^{-1},d]=\gamma_s.$$ \noindent Since $e_{k_i} = 1$, and
$d=1$ in the fundamental groups of $Y_n(G)$ and $X_n(G)$, because
$e_{k_i}$ are the vanishing cycles of the genus $2k' + n - 1$
fibrations on them, we obtain $\gamma_i = 1$ for any $i$. We can
easily write down the global monodromy of the genus $2k' + n -1$
Lefschetz fibration on $X_{n}(G)$: $\theta^{2} {\phi'}^{-1} \theta^2
\phi' = 1$, where  $\phi' = t_{a_{1}} \cdots t_{a_{k}} t_{a_{k+1}}
t_{b_{k+1}} \cdots t_{a_{k'}}t_{b_{k'}}t_{\gamma_1'}\cdots
t_{\gamma_s'}$. In conclusion, by taking $g=k'$, we proved

%\begin{Prop} \label{newpf} For each $n \geq 1$,  there is a closed symplectic $4$-manifold
%$X_n(G)$ with fundamental group $G$, which is obtained by a
%combination of Luttinger surgery and symplectic sum operations.
%\end{Prop}

%symplectic sum of $E(n)$ along a regular elliptic fiber with $M(G)$
%along a natural square zero torus, where $M(G)$ is a surface bundle
%over $\mathbb{T}^2$ constructed as a result of some Luttinger
%surgeries on a trivial symplectic bundle $\Sigma_g \times
%\mathbb{T}^2$.

%A straightforward generalization of Lemma~\ref{lef} based on
%Proposition~\ref{newpf} yields a new proof of
%Theorem~\ref{gompfdon}, with some additions.

\begin{Prop} \label{sect} Given any finitely presentable group $G$, there exists
a closed symplectic $4$-manifold $X_n(G)$ with fundamental group
$G$, which admits a genus $2g+n-1$ Lefschetz fibration over $S^2$
that has at least $4n + 4$ pairwise disjoint sphere sections of self
intersection $-2$. Moreover, $X_n(G)$ contains a homologically
essential embedded torus of square zero disjoint from these sections
which intersects each fiber of the Lefschetz fibration twice.
\end{Prop}

Note the genus $2g+n-1=2(k + l - s) + n-1$ of the Lefschetz
fibration on $X_n(G)$ given in Proposition~\ref{sect} depends on the
given presentation  $\langle x_{1}, \dots, x_{k} \ | \ r_{1}, \dots,
r_{s} \rangle$ of the finitely presentable group $G$. \\

 Now we are
ready to give a proof of Theorem~\ref{main2}:

\begin{proof} (\emph{of Theorem~\ref{main2}}):
For any integer $h \geq 2$, let $\mathcal{F}_h=\{K_{i} : i \in
\mathbb{N} \}$ denote an infinite family of genus $h$ fibered knots
in $S^3$ with pairwise distinct Alexander polynomials. Such families
of knots exist by the work of T. Kanenobu \cite{kan}. First, notice
that by Seifert-Van Kampen's theorem, $\pi_1 (X_n(G)_{K_i})= G$,
since all the loops on the torus $T$ given above are null-homotopic
in $X_n(G)$, and the homology class of a longitude of the knot $K_i$
in $S^3 \setminus \nu(K_i)$ is trivial.

Next, we show that ${X_n(G)}_{K_i}$ is homeomorphic to $X_n(G)$ for
any $K_{i} \in \mathcal{F}_h$. Recall that $E(n)$ contains a {\em
small} simply connected submanifold $N(n)$ with $b_{2} = 2$, called
a nucleus  \cite{Gom1}, whose complement is diffeomorphic to the
Milnor fiber of the  Brieskorn homology $3$-sphere $\Sigma(2, 3,
6n-1)$. Consequently, we see that $X_n(G)$ contains a copy of $N(n)$
and thus ${X_n(G)}_{K_i}$ contains a copy of an exotic nucleus
$N(n)_{K_{i}}$. Therefore we obtain the following decompositions:
$X_n(G) = N(n) \cup W(G, n, g)$ and ${X_n(G)}_{K_i} = N(n)_{K_{i}}
\cup W(G, n, g)$. Since the boundary of $N(n)_{K_i}$ is the
Brieskorn homology $3$-sphere, the argument which was elaborated in
details in \cite[page 12]{ako} shows that ${X_n(G)}_{K_i}$ is
homeomorphic to  $X_n(G)$ for any choice of $K_i$.

Now using the well-known knot surgery formula for Seiberg-Witten
invariants \cite{fs}, we see that the infinite family $\{
X_n(G)_{K_{i}}  : K_{i} \in \mathcal{F}_h\}$ consists of closed
symplectic $4$-manifolds which are pairwise non-diffeomorphic.

%Let $G$ denote a finitely presentable group. Then there exist an
%infinite family $\{ K_{i} : i \in \mathbb{N}\}$ of fibered knots of
%some fixed genus $h \geq 2$ with pairwise distinct Alexander
%polynomials such that the knot surgery to $X_n(G)$ along the above
%given torus $T$ produces an infinite family of pairwise
%non-diffeomorphic closed symplectic $4$-manifolds $\{X_n(G)_{K_i} :
%i  \in \mathbb{N}\}$ all homeomorphic to $X_n(G)$. Moreover,
%$\{X_n(G)_{K_i} : i  \in \mathbb{N}\}$ admits a genus $2g+2h+n-1$
%non-holomorphic Lefschetz fibration over $S^2$.

Observe that $X_n(G)_{K_i}$ admits a genus $2g + 2h + n -1$
Lefschetz fibration over $S^2$, which is induced from the genus
$2g+n-1$ Lefschetz fibration on $X_n(G)$. To complete the proof of
our theorem we need to show that the family $\{X_n(G)_{K_i}:K_{i}
\in \mathcal{F}_h\}$ can be chosen to consist of only non-complex
manifolds.

Since $X_n(G)$ obtained by a sequence of Luttinger surgeries from
the Kahler surface $E(n,g) := E(n) \#_{id} \Sigma_{g} \times
\mathbb{T}^2$ (where $n \geq 2$ and $g \geq 1$), and the Luttinger
surgery preserves the symplectic Kodaira dimension $\kappa^{s}$
\cite{HoLi}, we conclude that $\kappa^{s}(X_n(G))= \kappa^{s}
(E(n,g))$. Moreover, from the formula for Kodaira dimension under
the symplectic fiber sum it follows that $\kappa^{s} (E(n,g))=1$
(cf. \cite[page 350]{DZ}). In order to complete the proof, we use
the following facts: \\

(a) The complex Kodaira dimension $\kappa^{h}$ is equal to the
symplectic Kodaira dimension $\kappa^{s}$ for a smooth $4$-manifold
which admits a symplectic structure as well as a complex structure,
where these structures are not necessarily required to be compatible
\cite[Theorem 1.1]{DZ}.

(b) A complex surface with Kodaira dimension $1$ is properly
elliptic.

(c) The diffeomorphism type of an elliptic surface $S$ with $\chi(S)
> 0$ and $|\pi_{1}(S)|=\infty$  is determined by its
fundamental group (cf. \cite[Theorem 8.3.12]{gs}, and \cite{U}). \\

\emph{Case 1:} Suppose that $X_n(G)$ is a complex surface. By (b)
above, it follows that $X_n(G)$ is properly elliptic.
 Since the symplectic 4-manifolds $X_n(G)$ and $\{X_n(G)_{K_i} : K_i
\in \mathcal{F}_h\}$ have $b_2^+ \geq 2$ and $\chi \ne 0$, none of
the members of this family is a properly elliptic surface without
singular fibers.

If the fundamental group of $X_n(G)$ is infinite, then by comparing
the Seiberg-Witten invariants of $X_n(G)$ and $\{X_n(G)_{K_i} : K_i
\in \mathcal{F}_h\}$, which are all distinct, and using the fact
(c), we conclude that $\{X_n(G)_{K_i} : K_i  \in \mathcal{F}_h\}$
are not complex surfaces. Now suppose that the fundamental group of
$X_n(G)$ is finite, but not cyclic. According to the work of M. Ue
\cite{U} (see also \cite[Remark 8.3.13]{gs}), such an elliptic
surface does not admit any exotic smooth structure. This implies
that none of the smooth $4$-manifolds in $\{X_n(G)_{K_i} : K_i  \in
\mathcal{F}_h\}$ admits complex structures, since they are all
pairwise non-diffeomorphic and are exotic copies of $X_n(G)$.

Finally, assume that the fundamental group of $X_n(G)$ is finite
cyclic. Note that elliptic surfaces with finite cyclic fundamental
group consist of the following family: $E(n)_{p,q}$ ($1 \leq p \leq
q$), for which $\pi_{1} \cong \mathbb{Z}_{gcd(p,q)}$. We refer the
reader to \cite[Theorem 3.3.6]{gs} for the computation of the
Seiberg-Witten invariants of $E(n)_{p,q}$ (for complete details, see
original paper \cite{FS1}). Now, by carefully choosing the infinite
family  $\mathcal{F}_h=\{K_{i} : i \in \mathbb{N} \}$ of genus $h$
fibered knots in $S^3$ with pairwise distinct Alexander polynomials,
we can guarantee that the Seiberg-Witten invariants of each
$X_n(G)_{K_i}$ is different from the Seiberg-Witten invariants of
any $E(n)_{p,q}$. Thus, we conclude that members of  $\{X_n(G)_{K_i}
: K_i
\in \mathcal{F}_h\}$ are not complex. \\

\emph{Case 2:} Suppose that $X_n(G)$ is not a complex surface, but
there is at least one complex surface, say $X_n(G)_{K}$, which
belongs to the family $\{X_n(G)_{K_i} :  K_{i} \in \mathcal{F}_h\}$.
Then we can apply the above argument to $X_n(G)_{K}$, to obtain the
infinite family of non-complex symplectic 4-manifolds $\{ X_n(G)_{K,
K_{i}} : K_{i} \in \mathcal{F}_h\}$ via knot surgery on
$X_n(G)_{K}$. Now using the identification $X_n(G)_{K, K_{i}} =
X_n(G)_{K\# K_{i}}$ (see \cite{akah}), we finish the proof of the
theorem using the family of fibered knots $\{K\# K_{i}  :  K_{i} \in
\mathcal{F}_h\}$.

\end{proof}

\section{Exotic Stein fillings with arbitrary fundamental group}\label{ES}

In this final section we prove Theorem~\ref{main1}:

\begin{proof}(\emph{of Theorem~\ref{main1}}): We follow a strategy similar to the one we used in \cite{ako}
to prove our theorem. By Proposition~\ref{sect}, there is a closed
symplectic $4$-manifold $X_n(G)$ whose fundamental group is $G$,
which admits a Lefschetz fibration over $S^2$ that has  at least $4n
+ 4$ pairwise disjoint sphere sections of square $-2$. By removing a
neighborhood of all but one of these sections and a neighborhood of
a regular fiber we obtain a PALF (positive allowable Lefschetz
fibration) over $D^2$ which is a Stein filling of the contact
structure induced on its boundary (cf. \cite{ao1}). As shown in
\cite{ako}, the boundary $3$-manifold is a Seifert fibered
\emph{singularity link} and the induced contact structure is the
\emph{canonical} contact structure on this singularity link.

To produce an infinite family of exotic Stein fillings of the same
Seifert fibered singularity link with its canonical contact
structure, we use the family of Lefschetz fibrations obtained in the
proof of  Theorem~\ref{main2} via knot surgery. Similar approaches
were used to study the special cases in \cite{aems,ako}. The torus
we use for knot surgery is a torus that is induced from $\{pt\}
\times \mathbb{T}^2$ that descends to the surface bundle $M(G)$ as a
section and to the Lefschetz fibration $X_n(G) \to S^2$ as a
multi-section intersection each fiber twice.

We take an infinite family $\{ K_{i} : i \in \mathbb{N}\}$ of
fibered knots of some fixed genus $h \geq 2$ with pairwise distinct
Alexander polynomials and apply knot surgery to $X_n(G)$ along the
aforementioned torus, to produce an infinite family of pairwise
non-diffeomorphic closed $4$-manifolds $\{X_n(G)_{K_i} : i  \in
\mathbb{N}\}$ all homeomorphic to $X_n(G)$ as in the proof of
Theorem~\ref{main2}.

Moreover, $X_n(G)_{K_i}$ admits a genus $2g + 2h + n -1$ Lefschetz
fibration over $S^2$ induced from the genus $2g+n-1$ Lefschetz
fibration on $X_n(G)$, where the disjoint sphere sections of $X_n(G)
\to S^2$ extends as sections after the knot surgery.

As a consequence, by removing the neighborhood of some fixed number
of sections and a regular fiber of the Lefschetz fibration
$X_n(G)_{K_i} \to S^2$, we obtain an infinite family of Stein
fillings of the canonical contact structure on the boundary Seifert
fibered singularity link. Using the fact that any diffeomorphism of
the boundary neighborhood of a $-2$ sphere section and a regular
genus $2g+2k+n-1$ fiber extends and the Seiberg-Witten invariants of
${X_n(G)}_{K_i}$ are distinct, we see that our infinite family of
Stein fillings are not diffeomorphic.

Using the exact same argument as in the proof of \cite[Theorem
6]{ako} and Seifert-Van Kampen's Theorem, we deduce that all our
fillings have the same fundamental group $G$. The key point is that
the normal circles resulting from the removal of the $-2$ sphere
sections and the genus $2g+2h+n-1$ fiber are all nullhomotopic (see
\cite[page 8]{ako} for the details).

We claim that the Stein fillings that we constructed above are all homeomorphic---which
finishes the proof of our theorem. First, recall from Theorem 13 that
${X_n(G)}_{K_i}$ is homeomorphic to  $X_n(G)$ for any choice of $K_i$.

%First, recall that $E(n)$
%contains a {\em small} simply connected submanifold $N(n)$ with
%$b_{2} = 2$, called a nucleus  \cite{Gom1},  whose complement is
%diffeomorphic to the Milnor fiber of the  Brieskorn homology
%$3$-sphere $\Sigma(2, 3, 6n-1)$.  Consequently, we see that $X_n(G)$
%contains a copy of $N(n)$ and thus ${X_n(G)}_{K_i}$ contains a copy
%of an exotic nucleus $N(n)_{K_{i}}$. Therefore we obtain the
%following decompositions: $X_n(G) = N(n) \cup W(G, n, g)$ and
%${X_n(G)}_{K_i} = N(n)_{K_{i}} \cup W(G, n, g)$. Since the boundary
%of $N(n)_{K_i}$ is the Brieskorn homology $3$-sphere, the argument
%which was elaborated in details in \cite{ako} (see page 12) shows that ${X_n(G)}_{K_i}$ is
%homeomorphic to  $X_n(G)$ for any choice of $K_i$.

The knot surgery operation mostly affects the complement of the
removed neighborhoods of the regular genus $2g+n-1$ fiber and
$(-2)$-sphere sections. Let us make this more precise. We first
observe that in $X_n(G)$ the tubular neighborhoods of $(-2)$-sphere
sections are all disjoint from the cusp neighborhood of the torus
$T$ given above. Moreover, the tubular neighborhood of a regular
fiber intersects the cusp neighborhood along two disjoint copies of
$D^2 \times D^2$. Next, using the fact that our homeomorphism is
identity on the complement of the cusp neighborhood (see \cite[page
12]{ako} for the details), we can remove our configurations
entirely, except the two disjoint copies of $D^2 \times D^2$, by not
affecting the homeomorphism. Performing knot surgery operation on
$T$ changes these two disk bundles to $D^2 \times \Sigma(h,1)$,
where $\Sigma(h,1)$ denotes genus $h$ surface with one puncture.
Since any diffeomorphism of $\partial(D^2 \times \Sigma(h,1))$
extends, we can delete these two $D^2 \times \Sigma(h,1)$ as well
without affecting our homeomorphism.

\end{proof}

%{\Rem  The construction of the simply connected exotic Stein
%fillings and exotic Stein fillings with $\pi_1 = \mathbb{Z}\oplus
%\mathbb{Z}_m$ we described in \cite{ako} can be deduced from the
%constructions in this paper. We just simply set in Lemma \ref{pres}
%and Theorem \ref{arbit}, $p_i = q_i = 1$, for all $1 \leq i \leq g$,
%to obtain the simply connected fillings and $p_1 = 0$, $q_1 = m$,
%$p_i = q_i = 1$, for all $2 \leq i \leq g$, to obtain the fillings
%with $\pi_1 = \mathbb{Z}\oplus \mathbb{Z}_m$.}

\v  \noindent {\bf {Acknowledgement}}:  The work on this paper
started at the  FRG workshop ``Topology and Invariants of Smooth
4-manifolds" in Miami, USA and was mostly completed at the  FRG
workshop ``Holomorphic Curves and Low Dimensional Topology" held in
the Stanford University. The authors are very grateful to the
organizers of both workshops for creating a very stimulating
environment.  A. A. was partially supported by NSF grants
FRG-0244663 and DMS-1005741. B.O. was partially supported by the
Marie Curie International Outgoing Fellowship 236639.


\begin{thebibliography}{99999}

\bibitem{akah} M. Akaho, {\em A connected sums of knots and Fintushel-Stern knot surgery,} Turkish. J. Math. \textbf{20} (2006), 87--93.

\bibitem{ao1} S. Akbulut and B. Ozbagci, {\em Lefschetz fibrations on compact Stein surfaces}, Geom. Topol. {\bf 5} (2001), 319--334.

%\bibitem{ao2} S. Akbulut and B. Ozbagci, \emph{On the topology of compact Stein surfaces,} Int. Math. Res. Not. \textbf{15} (2002), 769--782.

\bibitem{Ak}  A. Akhmedov,
\emph{Construction of symplectic cohomology $\mathbb{S}^{2}\times \mathbb{S}^{2}$}, G\"{o}kova Geometry and Topology Proceedings, \textbf{14} (2007), 36--48.

\bibitem{ABP}  A. Akhmedov, R.I. Baykur and  B.D. Park,
\emph{Constructing infinitely many smooth structures on small 4-manifolds},
J. Topol. \textbf{1} (2008), no. 2, 409--428.

\bibitem{aems} A. Akhmedov, J.B. Etnyre, T.E. Mark and I. Smith, {\em A note on Stein fillings of contact manifolds}, Math. Res. Lett. 15 (2008), no. 6, 1127--1132.

\bibitem{ako}
A. Akhmedov and B. Ozbagci, \emph{Singularity links with exotic
Stein fillings}, arXiv 1206.2468v2.


\bibitem{AP}
A. Akhmedov, B.~D. Park,
\emph{Exotic Smooth Structures on Small 4-Manifolds with Odd Signatures},
Invent. Math. \textbf{181} (2010), no. 3, 577--603.


\bibitem{AZ} A. Akhmedov and W. Zhang, \emph{The fundamental group of symplectic 4-manifolds with $b_{2}^{+} = 1$}, preprint.


\bibitem{ABKP} J. Amor\'os, F. Bogomolov, L. Katzarkov and T. Pantev,
\emph{Symplectic Lefschetz fibrations with arbitrary fundamental groups},  J. Differential Geom. 54 (2000), no. 3, 489--545.

\bibitem{ADK}  D. Auroux, S. K. Donaldson and L. Katzarkov, \emph{Luttinger surgery along Lagrangian tori and non-isotopy
for singular symplectic plane curves}, Math. Ann. \textbf{326} (2003), 185--203.

%\bibitem{bd}
%F. A. Bogomolov and B.  de Oliveira, {\em Stein small deformations of strictly pseudoconvex surfaces}, Birational algebraic geometry (Baltimore, MD, 1996),  25--41, Contemp. Math., 207, Amer. Math. Soc., Providence, RI, 1997.


%\bibitem{bg} C. P. Boyer and K. Galicki, {\em Sasakian geometry}, Oxford Mathematical Monographs. Oxford University Press, Oxford, 2008.


%\bibitem{cnp} C. Caubel, A. N\'{e}methi, and P. Popescu-Pampu, \emph{Milnor open books and Milnor fillable contact $3$-manifolds},  Topology  45 (2006),  no. 3, 673--689.

\bibitem{don} S. K. Donaldson, {\em Lefschetz pencils on symplectic manifolds}, J.
Differential Geom. 53 (1999), no. 2, 205--236.

\bibitem{DZ} J. G. Dorfmeister, and W. Zhang, {\em The Kodaira dimension of Lefschetz fibrations},
Asian J. Math. \textbf{13} 2009, 341--358.


%\bibitem{eo} T. Etg\"{u} and B. Ozbagci, {\em Explicit horizontal open books on some plumbings}, Internat. J. Math. 17 (2006), no. 9, 1013--1031.


%\bibitem{et} J. Etnyre, {\em Lectures on open book decompositions and contact structures}, Floer homology, gauge theory, and low-dimensional topology, 103--141, Clay Math. Proc., 5, Amer. Math. Soc., Providence, RI, 2006.

\bibitem{FS1} R. Fintushel and R. Stern, \emph{Rational blowdowns of smooth $4$-manifolds}, J. Differential Geom. 46 (1997), 181--235.

\bibitem{fs} R. Fintushel and R. Stern,  \emph{Knots, links, and $4$
-manifolds}, Invent. Math. 134 (1998), no. 2, 363--400.


%\bibitem{f} M. H. Freedman, {\em The disk theorem for four-dimensional manifolds}, Proc. of Inter. Cong. of Math, Vol. 1, 2 (Warsaw, 1983) (Warsaw), PWN, 1984, pp. 647--663.

\bibitem{Gom1} R. E. Gompf, {\em  Nuclei of elliptic surfaces}, Topology 30 (1991), no. 3, 479--511.

\bibitem{Gom} R. E. Gompf, {\em A new construction of symplectic manifolds}, Ann. of Math. (2) 142 (1995), no. 3, 527--595.



\bibitem{gs}
R. E. Gompf and A. I. Stipsicz, {\em $4$--manifolds and Kirby
calculus}, Grad. Stud. Math., Vol. {\bf 20}, AMS, 1999.

\bibitem{Gu} Y. Gurtas, {\em Positive Dehn twist expressions for some elements of finite order in the mapping class group}, preprint, arXiv:math.GT/0404311.


\bibitem{HoLi} C.I. Ho and T.J. Li, {\em Luttinger surgery and Kodaira dimension}, Asian Journal of Math. 16(2) (2012), 299-318.



%\bibitem{grs} G. M. Greuel, and J. Steenbrink, {\em On the topology of smoothable singularities}, Proc. of Symp. in Pure Maths. 40 (1983), Part 1, 535--545.


%\bibitem{h} I. Hambelton, {\em Intersection forms, fundamental groups and 4-manifolds}, Proc. of Gokova. Geom. Topol. Conf. 15 (2008), 137--150. 2203--2223.


\bibitem{kan}

T. Kanenobu, \emph{Module d'Alexander des noeuds fibr\'{e}s et
polynôme de Hosokawa des lacements fibr\'{e}s}, (French) [Alexander
module of fibered knots and Hosokawa polynomial of fibered links]
Math. Sem. Notes Kobe Univ. 9 (1981), no. 1, 75--84.

%\bibitem{kk} J. Kollar and M. Kapovich, {\em Fundamental groups of links of isolated singularities}, JAMS, to appear.


\bibitem{kor} M. Korkmaz, {\em Lefschetz fibrations and an invariant of finitely
presented groups}, Int. Math. Res. Not. IMRN 2009, no. 9,
1547--1572.

%\bibitem{k} M. Korkmaz, {\em Noncomplex smooth 4-manifolds with Lefschetz fibrations}, Internat. Math. Res. Notices, 3 (2001), 115--128.

%\bibitem{ko} M. Korkmaz and B. Ozbagci, {\em On sections of elliptic fibrations}, Michigan Math. J. 56 (2008), 77-–87.

%\bibitem{kt} %Y. Kamishima and T. Tsuboi, {\em CR-structures on Seifert manifolds}, Invent. Math. 104 (1991), no. 1, 149--163.


%\bibitem{l} P. Lisca, {\em On symplectic fillings of lens spaces}, Trans. Amer. Math. Soc. 360 (2008), no. 2, 765--799 (electronic).

%\bibitem{lm} P. Lisca and G. Matic, {\em Transverse contact structures on Seifert 3-manifolds}, Algebr. Geom. Topol. 4 (2004), 1125--1144 (electronic).


\bibitem{Lu}  K. M. Luttinger, {\em Lagrangian tori in\/ $\mathbb{R}^4$},
J. Differential Geom. \textbf{42} (1995), 220--228.


%\bibitem{mat} Y. Matsumoto, {\em Lefschetz fibrations of genus two: a topological approach}, in Topology and Teichm\"{u}ller spaces (Katinkulta, Finland, 1995), World Sci. Publ., River Edge, NJ, 1996, 123--148.

%\bibitem{nem} A.  N\'{e}methi, \emph{On the canonical contact structure of links of complex surface singularities}, preprint


%\bibitem{npp} A.  N\'{e}methi and P. Popescu-Pampu, {\em On the Milnor fibres of cyclic quotient singularities}, Proc. Lond. Math. Soc. (3) 101 (2010), no. 2, 554--588.


%\bibitem{neu} W. Neumann, \emph{A calculus for plumbing applied to the topology of complex surface singularities  degenerating complex curves},  Trans. AMS 268 (2)  (1981), 299--344.

%\bibitem{oo} H. Ohta and  K. Ono, {\em Examples of isolated surface singularities whose links have infinitely many symplectic fillings}, J. Fixed Point Theory Appl. 3 (2008), no. 1, 51--56.

%\bibitem{o} B. Ozbagci, {\em Explicit horizontal open books on some Seifert fibered 3-manifolds}, Topology Appl. 154 (2007), no. 4, 908--916.

%\bibitem{os} B. Ozbagci and A.I. Stipsicz, {\em Contact 3-manifolds with infinitely many Stein fillings}, Proc. Amer. Math. Soc., 132 (2004), no. 5, 1549--1558.

%\bibitem{os1} B. Ozbagci and A.I. Stipsicz, {\em Noncomplex smooth $4$-manifolds with genus-$2$ Lefschetz fibrations}, Proc. Amer. Math. Soc., 128(2000), no. 10, 3125--3128.

\bibitem{t}
S. Tanaka, {\em On sections of hyperelliptic Lefschetz fibration},
Algebr. Geom. Topol.  12  (2012),  no. 4, 2259--2286.

\bibitem{U} M. Ue, {\em On the diffeomorphism types of elliptic surfaces with multiple fibers},
Invent. Math. (2006), 84  (1986), 633--643.

\bibitem{Us} M. Usher, {\em Minimality and symplectic sums},
Int. Math. Res. Not. (2006), Art. ID 49857, 17 pp.

\bibitem{Yu} K. Yun, {\em On the signature of a Lefschetz fibration coming from an involution}, Topology Appl. 153 (2006), no. 12, 1994--2012.


\end{thebibliography}
\end{document}